\documentclass[11pt]{amsart}
\usepackage{amsmath,amsthm,amsfonts,latexsym,amssymb,amscd}
\usepackage{graphicx}

\calclayout\evensidemargin .1in

\newcommand{\C}{\mathbb{C}}
\newcommand{\R}{\mathbb{R}}

\newtheorem{theorem}{Theorem}[section]
\newtheorem{lemma}[theorem]{Lemma}
\newtheorem{proposition}[theorem]{Proposition}
\newtheorem{corollary}[theorem]{Corollary}

\theoremstyle{definition}

\newtheorem{remark}[theorem]{Remark}

\newtheorem{definition}[theorem]{Definition}

\title[LEGENDRIAN REALIZATION IN CONVEX LEFSCHETZ FIBRATIONS AND CONVEX STABILIZATIONS]{LEGENDRIAN REALIZATION IN CONVEX LEFSCHETZ FIBRATIONS\\AND CONVEX STABILIZATIONS}

\author{Selman Akbulut}
\address{Department of Mathematics, Michigan State University, Lansing MI, USA}
\email{akbulut@math.msu.edu}
\thanks{The first author is partially supported by NSF FRG grant DMS- 0905917}

\author{M. Firat Arikan}
\address{Department of Mathematics, Middle East Technical University, Ankara, TURKEY}
\email{farikan@metu.edu.tr}
\thanks{The second author is partially supported by NSF FRG grant DMS-1065910, and also by TUBITAK grant }

\date{\today}

\begin{document}

\begin{abstract}

In this paper, we study compact convex Lefschetz fibrations on compact convex symplectic manifolds (i.e., Liouville domains) of dimension $2n+2$ which are introduced by Seidel and later also studied by McLean. By a result of Akbulut-Arikan, the open book on $\partial W$, which we call \emph{convex open book}, induced by a compact convex Lefschetz fibration on $W$ carries the contact structure induced by the convex symplectic structure (i.e., Liouville structure) on $W$. Here we show that, up to a Liouville homotopy and a deformation of compact convex Lefschetz fibrations on $W$, any simply connected embedded Lagrangian submanifold of a page in a convex open book on $\partial W$ can be assumed to be Legendrian in $\partial W$ with the induced contact structure. This can be thought as the extension of Giroux's Legendrian realization (which holds for contact open books) for the case of convex open books. Moreover, a result of Akbulut-Arikan implies that there is a one-to-one correspondence between convex stabilizations of a convex open book and convex stabilizations of the corresponding compact convex Lefschetz fibration. We also show that the convex stabilization of a compact convex Lefschetz fibration on $W$ yields a compact convex Lefschetz fibration on a Liouville domain $W'$ which is exact symplectomorphic to a \emph{positive expansion} of $W$. In particular, with the induced structures $\partial W$ and $\partial W'$ are contactomorphic.
\end{abstract}

\maketitle

%-----------------------------------------------------------------------------------------------
%-----------------------------------------------------------------------------------------------
%-----------------------------------------------------------------------------------------------

\section{Introduction}

Let $M$ be an oriented smooth $(2n+1)$-manifold. A \emph{contact structure} on $M$ is a global $2n$-plane field distribution $\xi$ which is totally non-integrable. Non-integrability condition is equivalent to the fact that locally $\xi$ can be given as the kernel of a 1-form $\alpha$ such that $\alpha \wedge (d\alpha)^n > 0$. If $\alpha$ is globally defined, then it is called a \emph{contact form} for $\xi$. An \emph{open book (decomposition)} $(B,\Theta)$ on $M$ is given by a codimension $2$ submanifold $B\hookrightarrow M$ with trivial normal bundle, and a fiber bundle $\Theta : M-B \rightarrow S^1$. The neighborhood of $B$ should have a trivialization $B \times
D^2$, where the angle coordinate on the disk agrees with the map $\Theta$. The manifold $B$ is called the \emph{binding}, and for any $t_0 \in S^1$ a fiber $X = \Theta^{-1}(t_0)$ is called a \emph{page} of the open book. A contact structure $\xi$ on $M$ is said to be \emph{supported by} (or \emph{carried by}) an open book $(B, \Theta)$ of $M$ if there exists a contact form $\alpha$ for $\xi$ such that
\begin{itemize}
\item[(i)] $(B, \alpha |_{TB})$ is a contact manifold.
\item[(ii)] For every $t\in S^1$, the page $X=\Theta^{-1}(t)$ is a symplectic manifold with symplectic form $d\alpha$.
\item[(iii)] If $\bar{X}$ denotes the closure of a page $X$ in $M$, then the orientation of $B$ induced by its contact form $\alpha |_{TB}$ coincides with its orientation as the boundary of $(\bar{X}, d\alpha)$.
\end{itemize}
An open book $(B, \Theta)$ on $M$ is called a \emph{contact open book} if it carries a contact structure on $M$. 

Open books can be defined also abstractly in the following way: If $(B,\Theta)$ is an open book for $M$, with pages $X$, and if $N(B)$ denotes the closed tubular neighborhood of $B$ in $M$, then $M -\textrm{int}N(B)$ is a fiber bundle over $S^1$ with fibers $\bar{X}$, hence it is diffeomorphic to the
\emph{(generalized) mapping torus} $\bar{X}(h):= [0,1] \times \bar{X} / \sim$, where $\sim$ is the identification $(0, p) \sim (1,h(p))$ for some diffeomorphism $h: \bar{X} \to \bar{X} $ which is identity near $\partial \bar{X} \approx B$. Observe that $\partial \bar{X}(\phi)\approx B \times S^1$, and so there is a unique way to glue $N(B)\approx B \times D^2$  and get $M$ back. Therefore, the pair $(\bar{X}, h)$ (abstractly) determines the open book $(B,\Theta)$, and it is called an \emph{abstract open book} of $M$. The map $h$ is called the \emph{monodromy} of the (abstract) open book.\\

Our main theorem is Theorem \ref{thm:Legendrian_Realiz_in_convex_Lef_fib} stated below. It can be considered as the extension of the following Giroux's result (given for contact open books) when the given contact open book arises as the boundary open book of a compact convex Lefschetz fibration over the $2$-disk $D^2$. 

\begin{theorem} [See, for instance, Lemma 4.2 in \cite{V}] \label{thm:Legendrian_Realiz_in_Contact_open_book}
Let $(X,h)$ be a contact open book for the contact manifold $(M,\xi=\emph{Ker}(\alpha))$ of dimension $2n+1$ for $n\geq2$. Let $L$ be any Lagrangian $n$-sphere embedded in a page $(X,d\alpha)$. Then there exist an isotopy $\xi_t$, $t \in [0,1]$, of contact structures with $\xi_0=\xi$ and a contact open book carrying $\xi_1$ with symplectomorphic page and isotopic monodromy such that $L$ is Legendrian in $(M,\xi_1)$.
\end{theorem}

A \emph{compact convex symplectic manifold} (or a \emph{Liouville domain}) is a pair $(W,\Lambda)$ where $W^{2n+2}$ is a compact manifold with boundary, together with a \emph{convex symplectic structure} (or a \emph{Liouville structure}) which is a one-form $\Lambda$ on $W$ such that $\Omega=d\Lambda$ is symplectic and the $\Omega$-dual vector field $Z$ of $\Lambda$ defined by $\iota_Z \Omega=\Lambda$ (or, equivalently, $\mathcal{L}_{Z}\Omega=\Omega$), where $\iota$ denotes the interior product and $\mathcal{L}$ denotes the Lie derivative, should point strictly outwards along $\partial W$. Since $\Omega$ and $Z$ (resp. $\Omega$ and $\Lambda$) together uniquely determine $\Lambda$ (resp. $Z$), one can replace the notation with the triple $(W,\Omega,Z)$ (resp. $(W,\Omega,\Lambda)$). The one-form $\Lambda_{\partial}:=\Lambda|_{\partial W}$ is contact (i.e., $\Lambda_{\partial} \wedge (d\Lambda_{\partial})^{n}>0$), and the contact manifold $(\partial W,\textrm{Ker} (\Lambda_{\partial}))$ is called the \emph{convex boundary} of $(W,\Lambda)$. The \emph{completion} of $(W,\Omega,Z)$ is obtained from $W$ by gluing the positive part $\partial W \times [0,\infty)$ of the symplectization $(\partial W \times \mathbb{R}, d(e^t \Lambda_{\partial}))$ of its convex boundary. A \emph{compact convex symplectic deformation} (or a \emph{homotopy of Liouville domains}) is a smooth $1$-parameter family $(W,\Omega_t,Z_t)$, $t \in [0,1]$, of compact convex symplectic manifolds.

\begin{definition} \label{def:Lefschetz_Fibration}
A smooth map $\pi: W^{2n+2} \rightarrow \Sigma^2$ is called a (\emph{smooth}) \textit{Lefschetz fibration} if $\pi$ is onto and has finitely many critical points such that there exist complex coordinates $w=(z_1, z_2, ...,z_{n+1})$ around each critical point $p \in W$ and a complex coordinate $z$ around the corresponding critical value $\pi(p)\in\Sigma$ which are consistent with the given orientation of $W$ and $\Sigma$, and $\pi(p)$ has the local representation:
$$z=\pi(w)=z_0+z_1^2 + z_2^2 +...+z_{n+1}^2.$$
\end{definition}

By definition, $\pi$ has finitely many critical values. Let $a$ be a regular value and $X=\pi^{-1}(a)$ a regular fiber. Being an honest fibration away from critical values, all regular fibers are diffeomorphic to $X$. When $\Sigma$ is a two-disk $D^2 \subset \C$, then $\pi$ determines (induces) an open book on the boundary $\partial W$. The \emph{monodromy} of $\pi$ is defined to be the monodromy of the induced open book $(X,h)$ on $\partial W$. We will denote such Lefschetz fibration over $D^2$ by the quadraple $(\pi,W,X,h)$.\\

In \cite{S1,S2}, Seidel introduced ``exact Lefschetz fibrations'' which are further studied by McLean \cite{Mc} under the name ``compact convex Lefschetz fibrations''. By definition the fiber and the base of such fibration $\pi$
are compact connected manifolds with boundary, some triviality condition is satisfied near the horizontal boundary component, and the total space $W$ is equipped with an exact $2$-form $\omega=d\lambda$ whose restriction on a regular fiber $X$ defines a convex symplectic (Liouville) structure, and $\pi$ has finitely many critical points of complex quadratic type (see Section \ref{sec:Com_con_Lefschetz_fib} for precise statements). By adding a pullback of a positive volume form on the base surface $S$, $\omega$ gives a symplectic form $\Omega$ on $W$ \cite{S2}. Indeed, one can construct a convex symplectic structure $\Lambda$ on $W$ with $d\Lambda=\Omega$ which restricts to the convex symplectic structure on each regular fiber of $\pi$ \cite{Mc}. A convex symplectic structure $(\Omega,\Lambda)$ constructed in this way is said to be \emph{compatible with} the compact convex Lefschetz fibration $\pi$ (see \cite{AA}). For our purpose we will only consider compact convex Lefschetz fibrations over $D^2$. Also, in our notation, we would like to specify the regular fiber and the monodromy of $\pi:W \to D^2$ as well as a convex symplectic structure $(\Omega,\Lambda)$ on $W$ compatible with $\pi$. Therefore, we let $(\pi,W,\Omega,\Lambda,X,h)$ denote the compatible compact convex Lefschetz fibration on the Liouville domain $(W,\Omega,\Lambda)$ over $D^2$ with the following properties:

\begin{itemize}
\item[(i)] Underlying smooth Lefschetz fibration is $(\pi,W,X,h)$. 
\item[(ii)] The restriction $\Lambda |_X$ defines a convex symplectic structure on every regular fiber $X$.
\item[(iii)] The monodromy $h$ of $\pi$ is a symplectomorphism of $(X,\Omega|_X)$
\end{itemize}

We'll call an open book $(X,h)$ on $\partial W$ a \emph{convex open book} if it is induced by a compatible compact convex Lefschetz fibration $(\pi,W,\Omega,\Lambda,X,h)$. By a \emph{compatible compact convex Lefschetz deformation}, we mean a smooth $1$-parameter family $(\pi_t,W,\Omega_t,\Lambda_t,X_t,h_t)$, for $t \in [0,1]$, of compatible compact convex Lefschetz fibrations on $W$. We are now ready to state our main theorem:

\begin{theorem} \label{thm:Legendrian_Realiz_in_convex_Lef_fib}
Let $(\pi,W,\Omega,\Lambda,X,h)$ be any compatible compact convex Lefschetz fibration. Let $L \subset X_0 \setminus \partial X_0$ be any simply connected Lagrangian submanifold embedded in a page $X_0$ of the boundary convex open book $(X,h)$. Then there exists another compatible compact convex Lefschetz fibration $(\pi, W,\Omega',\Lambda',X,h)$ on $W$ such that  
\begin{itemize}
\item[(i)] there is a compact convex symplectic deformation $(\Omega_t,\Lambda_t)$,  $t \in [0,1]$, on $W$ connecting $(\Omega_0,\Lambda_0)=(\Omega,\Lambda)$ to $(\Omega_1,\Lambda_1)=(\Omega',\Lambda')$,
\item [(ii)] the completions of $(W,\Omega,\Lambda)$ and $(W,\Omega',\Lambda')$ are exact symplectomorphic,
\item [(iii)] $(\pi, W,\Omega_t,\Lambda_t,X,h)$,  $t \in [0,1]$, is a compatible compact convex Lefschetz deformation, %connecting $(\pi, W,\Omega,\Lambda,X,h)$ to $(\pi, W,\Omega',\Lambda',X,h)$,
\item [(iv)] $\xi_t=\emph{Ker}(\Lambda_t|_{\partial W})$ is an isotopy of contact structures on $\partial W$ from  $\xi=\emph{Ker}(\Lambda|_{\partial W})$ to $\xi'=\emph{Ker}(\Lambda'|_{\partial W})$, and
\item [(v)] the submanifold $L$ is Legendrian in $(\partial W,\xi')$.
\end{itemize}
\end{theorem}

\begin{remark}
 The only topological property of $S^n$ used in the proof of Theorem \ref{thm:Legendrian_Realiz_in_Contact_open_book} is the fact that $H_{dR}^1(S^n)=0$ (as $n\geq2$). This is why we state our main theorem for simply connected Lagrangian submanifolds. We remark that the same generalization also applies to Theorem \ref{thm:Legendrian_Realiz_in_Contact_open_book}. We also note that according to the terminology of \cite{S3} the Liouville domains $(W,\Omega_t,\Lambda_t)$ above are mutually Liouville isomorphic. 
 % A \emph{Liouville isomorphism} between two Liouville domains $(W,\Lambda_i)$ ($i=0,1$) is a diffeomorphism $\varphi : \widehat{W_0} \to \widehat{W_1}$ satisfying $\varphi^*\Lambda_1 = \Lambda_0 + d(some compactly supported function)$.
\end{remark}

The proof will be given in Section \ref{sec:the_proof}. As an application of this result one can easily conclude

\begin{corollary} \label{cor:any_stab_can_be_made_convex}
After a compatible compact convex Lefschetz deformation $(\pi,W,\Omega_t,\Lambda_t,X,h)$, any positive stabilization of a compatible compact convex Lefschetz fibration $(\pi,W,\Omega_0,\Lambda_0,X,h)$ can be realized as a convex stabilization of $(\pi,W,\Omega_1,\Lambda_1,X,h)$. \qed 
\end{corollary}

Convex stabilizations are special types of positive stabilizations where stabilizations are along properly embedded Legendrian disks on pages of boundary convex open books. These are originally introduced for compatible exact Lefschetz fibrations in \cite{AA} which are exactly what we call compatible compact convex Lefschetz fibrations in the present paper. We omit the proof of Corollary \ref{cor:any_stab_can_be_made_convex} as its being a direct consequence of definitions and Theorem \ref{thm:Legendrian_Realiz_in_convex_Lef_fib}. The result of this paper should be thought of a partial generalization of \cite{AO}, where in dimension four a surjection $ \left\{PALF's \right\} \longrightarrow \left\{ \mbox{Stein Manifolds} \right\} $ was established, here in any dimension we show the part of that result, which is an analogous map is well defined. 

In Section \ref{sec:convex_stabilization}, we will recall definitions and some properties of convex stabilizations for compatible compact convex Lefschetz fibrations and convex open books. In particular, we will give the proof of the fact that any convex stabilization of a compatible compact convex Lefschetz fibration on $W$ yields a compatible compact convex Lefschetz fibration on a Liouville domain $W'$ which is exact symplectomorphic to a positive expansion (see below) of $W$, and, in particular, with the induced structures $\partial W$ and $\partial W'$ are contactomorphic.

%==================================================================
\medskip \noindent {\em Acknowledgments.\/} The authors would like to
thank Yasha Eliashberg for helpful conversations and remarks, and also Max Planck Institute for Mathematics for their hospitality.
%==================================================================

%-----------------------------------------------------------------------------------------------
%-----------------------------------------------------------------------------------------------
%-----------------------------------------------------------------------------------------------

\section{Complete convex symplectic (Liouville) manifolds}

We first recall complete convex symplectic manifolds and their deformations. They are also known as Liouville manifolds. The notions below can be found in \cite{EG}, \cite{CE}, \cite{S3}, and \cite{Mc}. An \emph{exact} \emph{symplectic manifold} is a manifold $W$ together with a symplectic form $\Omega$ and a $1$-form $\Lambda$ satisfying $\Omega=d\Lambda$. In such a case there is a global \emph{Liouville vector field} $Z$ of $\Omega$ defined by the equation $\iota_{Z} \Omega=\Lambda$. We will write exact symplectic manifolds as triples of the form $(W,d\Lambda,Z)$, or $(W,d\Lambda,\Lambda)$, or just as pairs $(W,\Lambda)$. A smooth map $\phi:(W_0,\Lambda_0)\to (W_1,\Lambda_1)$ between two exact symplectic manifolds is called an \emph{exact symplectomorphism} if $\phi^*(\Lambda_1)-\Lambda_0$ is exact.

\begin{definition}[\cite{CE}, \cite{S3}] \label{def:Liouville_manifold}
A \emph{complete convex symplectic manifold} (\emph{Liouville manifold}) is an exact symplectic manifold $(W,d \Lambda, Z)$ such that
\begin{itemize}
\item[(i)] the expanding vector field $Z$ is \emph{complete} (i.e., its flow exists for all times), and
\item[(ii)] there exists an exhaustion $W=\bigcup_{k=1}^{\infty} W^k$ by compact domains $W^k \subset W$ with smooth boundaries along which $Z$ points outward.
\end{itemize}
If $Z^{-t}:W \rightarrow W$  $(t>0)$ denotes the contracting flow of $Z$, then the \emph{core} (or \emph{skeleton}) of the Liouville manifold $(W,d \Lambda, Z)$ is defined to be the set
$$\textrm{Core}(W,\Lambda)=\textrm{Core}(W,d \Lambda,Z):=\bigcup_{k=1}^{\infty} \bigcap_{t>0} Z^{-t}(W^k).$$ A Liouville manifold is of \emph{finite type} if its core is compact. A \emph{Liouville cobordism} $(W,d \Lambda,Z)$ is a compact cobordism $W$ with an exact symplectic structure $(d\Lambda,Z)$ such that $Z$ points outwards along $\partial_+ W$ and inwards along $\partial_- W$. A Liouville cobordism with $\partial_- W=\emptyset$ is called a \emph{Liouville domain}.(One can see that this definition of Lioville domains is equivalent to the one given above.)
\end{definition}

\begin{remark} \label{rem:Collar_is_Symplectization}
Let $(W,d \Lambda,Z)$ be a finite type Liouville manifold and $V\subset W$ a compact domain containing the core with smooth boundary $M=\partial V$. Then the whole symplectization $(M \times \mathbb{R}, d(e^t \Lambda|_M))$ symplectically embeds into $W$. The contact manifold $(M,\Lambda|_M)$ is canonical and called \emph{the ideal contact boundary}  of $(W,d \Lambda,Z)$. Similarly, if $M$ is the convex boundary $\partial_+W (=\partial W)$ of a Liouville domain $(W,d \Lambda,Z)$, then the negative half of the symplectization $(M \times \mathbb{R}, d(e^t \lambda|_M))$ symplectically embeds into $W$ (as a collar neighborhood of $M$ in $W$) so that its complement in $W$ is $\textrm{Core}(W,d\Lambda,Z)$ and the embedding matches the positive $t$-direction of $\mathbb{R}$ with $Z$. Therefore, the completion of a Liouville domain is a finite type Liouville manifold. 
\end{remark}

Recall from \cite{CE} that a smooth family $(W,\Omega_t,Z_t)$, $t \in [0,1]$, of Liouville manifolds is called a \emph{simple Liouville homotopy} if there exists a smooth family of exhaustions $W = \cup_{k=1}^{\infty} W_t^k$ by compact domains $W_t^k \subset W$ with smooth boundaries along which $Z_t$ is outward pointing. A smooth family $(W,\Omega_t,Z_t)$, $t \in [0,1]$, of Liouville manifolds is called \emph{Liouville homotopy} if it is a composition of finitely many simple homotopies.

\begin{lemma} [\cite{CE}, Lemma 11.6 and Proposition 11.8] \label{lem:Liouville_homotopy}
A smooth family $(W,\Omega_t,Z_t)$, $t \in [0, 1]$, of Liouville manifolds of finite type is a Liouville homotopy if the closure $\overline{\cup_{t \in [0,1]}\emph{Core}(W,\Omega_t,Z_t)}$ of the union of their cores is compact. In particular, any smooth family of Liouville domains determines a Liouville homotopy of their completions. Moreover, homotopic Liouville manifolds are exact symplectomorphic. In particular, homotopic Liouville domains have exact symplectomorphic completions.
\end{lemma}

%-----------------------------------------------------------------------------------------------
%-----------------------------------------------------------------------------------------------
%-----------------------------------------------------------------------------------------------

Next we define the positive expansion (which will be used in the last section) of a Liouville domain for any given smooth strictly positive real-valued function on its boundary. Given a compact manifold with boundary and a strictly positive smooth function $f:\partial W \to \R$, let $W^{cob}_f$ denote the set 
$$W^{cob}_f:=\{(x,t)\;|\; x \in \partial W, \; 0\leq t \leq f(x) \}\subset \partial W \times [0,\infty)$$
which is the region in $W \times [0,\infty)$ under the graph
$$\textrm{Graph}(f):=\{(\,x,f(x) \,)\;|\; x \in \partial W \}.$$

\begin{definition} \label{def:f-positive_Weinstein_expansion}
Let $(W,\Omega,Z)$ be a Liouville domain and $f:\partial W \to \R$ any strictly positive smooth function. The \emph{positive $f$-expansion} of $(W,\Omega,Z)$ is a Liouville domain, denoted by $(W,\Omega,Z)_f,$ obtained by gluing the Liouville cobordism 
$$(W^{cob}_f, d(e^t (\iota_Z \Omega)|_{\partial W}), \partial / \partial t)$$ (between $\partial_-W^{cob}_f=\partial W \times 0$ and $\partial_+W^{cob}_f=\textrm{Graph}(f)$)  to $W$ via the map
$\partial W \to \partial_-W^{cob}_f$ given by $x \mapsto (x,0)$.
\end{definition}

Note that by construction $(W,\Omega,Z)$ has a codimension zero Liouville embedding into $(W,\Omega, Z)_f$. Also $(W,\Omega,Z)_f$ is diffeomorphic to $W$ since $W^{cob}_f$ is a smooth trivial cobordism (i.e., diffeomorphic to the product $\partial W \times [0,1]$) for any strictly positive smooth function $f$. %For future reference let us remember this fact as a lemma.

%\begin{lemma} \label{prop:k-positive_expansion_diffeo_to_W}
%$(W,\Omega,Z)_f$ is diffeomorphic to $W$. \qed
%\end{lemma}

\begin{remark} \label{rem:no_symplecto_but_contacto}
Note that a positive $f$-expansion $W'$ of a Liouville domain $W$ is not symplectomorphic to $W$. However, $W'$ and $W$ have contactomorphic boundaries. To see this, note that $\partial W=\partial_-W^{cob}_f=\partial W \times 0$ and $\partial W'=\partial_+W^{cob}_f=\textrm{Graph}(f)$ are cobordant via the Liouville cobordism  $(W^{cob}_f, d(e^t (\iota_Z \Omega)|_{\partial W}), \partial / \partial t)$. The contact structure $\xi$ (resp. $\xi'$) on $\partial W$ (resp. $\partial W'$) is the kernel of the contact form $\alpha=e^0 (\iota_Z \Omega)|_{\partial W}$ (resp. $\alpha'=e^k(\iota_Z \Omega)|_{\partial W'}$). Also by following the trajectories of the Liouville vector field $\partial / \partial t$, we obtain a diffeomorphism $F:\partial W=\partial W \longrightarrow \partial W'$. Then we compute that $F^*(\alpha')=e^f \alpha$ (so $F_*$ maps $\xi$ to $\xi'$). Hence, $(\partial W,\xi)$ and $(\partial W',\xi')$ are contactomorphic.
\end{remark}

%-----------------------------------------------------------------------------------------------
%-----------------------------------------------------------------------------------------------
%-----------------------------------------------------------------------------------------------

\section{Compact convex Lefschetz fibrations} \label{sec:Com_con_Lefschetz_fib}

Exact (compact convex) Lefschetz fibrations are introduced in \cite{S1,S2} (see also \cite{Mc}).
Let $\pi:W^{2n+2} \to S$ be a differentiable fiber bundle, denoted by $(\pi,W)$, whose fibers and base
are compact connected manifolds with boundary. The boundary of such an $W$ consists of two parts: The vertical part $\partial_v W;=\pi^{-1}(\partial S)$, and the horizontal part $\partial_h W:=\bigcup_{z \in S} \partial W_z$ where $W_z=\pi^{-1}(z)$ is the fiber over $z \in S$. 

\begin{definition} [\cite{S1,S2}] \label{def:Exact_Symplectic_Fibration}
An \emph{compact convex symplectic fibration} $(\pi,W,\omega,\lambda)$ over a bordered surface $S$ is a differentiable fiber bundle $(\pi,W)$ equipped with a $2$-form $\omega$ and a $1$-form $\lambda$ on $W$, satisfying $\omega=d\lambda$, such that
\begin{itemize}
\item[(i)] each fiber $W_z$ with $\omega_z=\omega|_{W_z}$ and $\lambda_z=\lambda|_{W_z}$ is a compact convex symplectic manifold,
\item[(ii)] the following triviality condition near $\partial_h W$ is satisfied: Choose a point $z \in S$ and consider the trivial fibration $\tilde{\pi} : \tilde{W}:= S \times W_z \to S$ with the forms $\tilde{\omega}, \tilde{\lambda}$ which are pullbacks of  $\omega_z,\lambda_z$, respectively. Then there should be a fiber-preserving diffeomorphism $\Xi:N \to \tilde{N}$ between neighborhoods $N$ of $\partial_h W$ in $W$ and $\tilde{N}$ of $\partial_h \tilde{W}$ in $\tilde{W}$ which maps $\partial_h W$ to $\partial_h \tilde{W}$, equals the identity on $N \cap W_z$, and $\Xi ^*\tilde{\omega}=\omega$ and $\Xi ^* \tilde{\lambda}=\lambda$.
\end{itemize}
\end{definition}

\begin{definition} [\cite{S1,S2}] \label{def:Exact_Lefschetz_Fibration}
A \emph{compact convex Lefschetz fibration} is a tuple $(\pi,W,S,\omega,\lambda,J_0,j_0)$ which satisfies the following conditions:
\begin{itemize}
\item[(i)] $\pi:W \to S$ is allowed to have finitely many critical points all of which lie in the interior of $W$.
\item[(ii)] $\pi$ is injective on the set $C$ of its critical points.
\item[(iii)] $J_0$ is an integrable complex structure defined in a neighborhood of $C$ in $W$ such that $\omega$ is a K\"ahler form for $J_0$.
\item[(iv)] $j_0$ is a positively oriented complex structure on a neighborhood of the set $\pi(C)$ in $S$ of the critical values.
\item[(v)] $\pi$ is $(J_0,j_0)$-holomorphic near $C$.
\item[(vi)] The Hessian of $\pi$ at any critical point is nondegenerate as a complex quadratic form, in other words, $\pi$ has nondegenerate complex second derivative at each its critical point.
\item[(vii)]  $(\pi,W\setminus \pi^{-1}(\pi(C)),\omega,\alpha)$ is a compact convex symplectic fibration over $S \setminus \pi(C)$.
\end{itemize}
\end{definition}

\begin{remark} \label{rem:Exact_Lefschetz_Fibration}
For the codimension two corners (i.e., $\partial_v W \cap \partial_h W$) of the total space $W$  there is canonical way of smoothening. Assuming such smoothening has been made, all total spaces will be assumed to be smooth through out the paper. Also the statements (ii)-(vi) guarantees that the singularities of $\pi$ are of Lefschetz type as in Definition \ref{def:Lefschetz_Fibration}. Let us consider a pair $(J,j)$ where $J$ is an almost complex structure on $W$ agreeing with $J_0$ near $C$ and $j$ is a positively oriented complex structure on $S$ agreeing with $j_0$ near $\pi(C)$ such that $\pi$ is $(J,j)$-holomorphic and $\omega(\cdot,J \cdot)|_{\textrm{Ker}\,(\pi_*)}$ is symmetric and positive definite everywhere. As pointed out in \cite{S1}, the space of such pairs is always contractible, and in particular, always nonempty, and furthermore, once we fixed $(J,j)$, we can modify $\omega$ by adding a positive $2$-form on $S$ so that it becomes symplectic and tames $J$ everywhere on $W$.
\end{remark}

\begin{theorem} [\cite{Mc}] \label{thm:Mc} 
Suppose that the base surface of a compact convex Lefschetz fibration $(\pi,W,S,\omega,\lambda,J_0,j_0)$ is a compact convex symplectic manifold $(S,\lambda_S)$. Then there exists a constant $K > 0$ such that for all $k \geq K$ we have $\Omega := \omega + k\pi^*(d\lambda_S)$ is a symplectic form, and the $\Omega$-dual $Z$ of $\Lambda:=\lambda + k\pi^*(\lambda_S)$ is transverse to $\partial W$ and pointing outwards. (In other words, $(W,\Omega,\Lambda)$ is a compact convex symplectic manifold, i.e., a Liouville domain.)
\end{theorem}

\begin{definition} [\& Notation](\cite{AA}) \label{def:compatible_exact_mfld_fib}
A compact convex symplectic manifold $(W,\Omega,\Lambda)$ and a compact convex Lefschetz fibration $(\pi,W,S,\omega,\lambda,J_0,j_0)$ are said to be \emph{compatible} if for a pair $(J,j)$ as in Remark \ref{rem:Exact_Lefschetz_Fibration} there exists a positive volume form $\omega_S$ on $S$ such that  $$\Omega=\omega + \pi^*(\omega_S)$$ and $\Omega$ tames $J$ everywhere on $E$. From now on, we'll always take $S=D^2$, and assuming the above choice of $(J,j)$ is already made, $(J_0,j_0)$ will be dropped from the notation. We will denote a compatible compact convex Lefschetz fibration (over $D^2$) on the compact convex symplectic manifold $(W,\Omega,\Lambda)$ by the tuple $(\pi,W,\Omega,\Lambda,X,h)$ as stated in the introduction.
\end{definition}

\begin{definition} [\cite{AA}]
An open book induced by a compatible compact convex Lefschetz fibration over the disk $D^2$ is called a \emph{convex} (\emph{exact}) \emph{open book}.
\end{definition}

As mentioned before compatible compact convex Lefschetz fibrations are studied under the name ``compatible exact Lefschetz fibrations'' in \cite{AA}. Therefore, the following theorem is the restatement of Theorem 4.7. of \cite{AA} in the terminology of the present paper.

\begin{theorem} \label{thm:Liouville_Openbooks_Support}
The convex open book $(X,h)$ induced by a compatible compact convex Lefschetz fibration $(\pi,W,\Omega,\Lambda,X,h)$ caries the induced contact structure $\emph{Ker}(\Lambda|_{\,\partial W})$ on $\partial W$. \qed
\end{theorem}

%-----------------------------------------------------------------------------------------------
%-----------------------------------------------------------------------------------------------
%-----------------------------------------------------------------------------------------------

\vspace{.05in}
\section{The proof of Theorem \ref{thm:Legendrian_Realiz_in_convex_Lef_fib}} \label{sec:the_proof}

The proof of the theorem will make use of the following proposition.

\begin{proposition} \label{prop:Extending_the_homotopy}
Let $(\pi,W,\Omega,\Lambda,X,h)$ be any compatible compact convex Lefschetz fibration, and denote by $(d\alpha_0,\alpha_0)$ the restriction of the Liouville structure $(\Omega,\Lambda)$ to the fixed page $X_0 \approx X$ of the boundary convex open book $(X,h)$. Let $G_t=(X_0,\alpha_{t})$, $t \in [0,1]$, be any homotopy of Liouville domains on $X_0$.
%which is constant near $\partial X_0$. %such that the symplectic structure $\omega_0$ is fixed for every time $t$.
Then there exists a compatible compact convex Lefschetz deformation $(\pi,W,\Omega_t,\Lambda_t,X,h)$, $t \in [0,1]$, such that the following hold\\
\begin{itemize}
\item [(i)] The smooth $1$-parameter family $H_t=(\Omega_t,\Lambda_t)$ on $W$ defines a homotopy of Liouville domains from $(\Omega_0,\Lambda_0)=(\Omega,\Lambda)$ to another Liouville structure $(\Omega_1,\Lambda_1)=(\Omega',\Lambda')$. \\%, which is compactly supported in a neighborhood $N$ of $X_0$ in $W$. \\
\item [(ii)]$\xi_t=\emph{Ker}(\Lambda_t|_{\partial W})$ defines an isotopy of contact structures on $\partial W$ from $\xi_0=\emph{Ker}(\Lambda|_{\partial W})$ to $\xi_1=\emph{Ker}(\Lambda'|_{\partial W})$.\\
\item [(iii)] For each $t \in [0,1]$ the convex open book induced by $(\pi,W,\Omega_t,\Lambda_t,X,h)$ carries $\xi_t$.\\
\item[(iv)] $G_t=H_t|_{X_0}$. That is, the restriction of $H_t$ to the page $X_0$ is equal to $G_t$.
\end{itemize}
\end{proposition}

\begin{proof} Without loss of generality let $\pi$ fiber over the unit disk $D^2=\{z\in \C\, | \, |z| \leq 1\}$. By definition there is a compact convex Lefschetz fibration $(\pi,W,D^2,\omega,\lambda,J_0,j_0)$ such that $$\Lambda=\lambda + K\pi^*(\lambda_{D^2})$$ for some $K>0$ where $\lambda_{D^2}$ is some convex symplectic structure on $D^2$. Note that $\pi$ has no singular fibers in a neighborhood of $X_0$ in $W$. Say $X_0=\pi^{-1}(e^{i\theta_0})$ for some $e^{i\theta_0}\in S^1$. For $0<\epsilon$, $0<\delta<\pi$, let us write $l(\theta_0,\delta)$ for the arc $\{e^{i\theta} \; | \; \theta_0-\delta\leq \theta \leq \theta_0+\delta\} \subset S^1$. Also let $A_{\theta_0}(\epsilon,\delta)$ denote the region in $D^2$ given by $$A_{\theta_0}(\epsilon,\delta)=\{re^{i\theta} \in \C  \; | \; 1-\epsilon \leq r \leq 1, \quad e^{i\theta} \in l(\theta_0,\delta)\}.$$ Choose $\epsilon,\delta$ small enough so that $\pi$ has no singular fibers over $A_{\theta_0}(\epsilon,\delta)$. Then, by Lemma 1.1 in \cite{S2}, there is a local trivialization for $\pi$ around $X_0$. More precisely, there exists a fiber-preserving diffeomorphism $\Psi: A_{\theta_0}(\epsilon,\delta) \times X_0 \longrightarrow \pi^{-1}(A_{\theta_0}(\epsilon,\delta))$ such that
$$\Psi |_{\{0\} \times X_0} = \textrm{id} \quad \textrm{and} \quad
\Psi^*\lambda = \alpha_0 + f_1 dx_1 + f_2 dx_2+ dR$$
where $f_1,f_2,R \in C^{\infty}(A_{\theta_0}(\epsilon,\delta) \times X_0,\R)$ are functions which vanish
near $A_{\theta_0}(\epsilon,\delta) \times \partial X_0$ and $x_1,x_2$ are real coordinates on $D^2$.
Denote by $X_{r,\theta}$ the fiber $\pi^{-1}(re^{i\theta} )=\Psi(re^{i\theta} \times X_0)\subset W$ over the point $re^{i\theta} \in A_{\theta_0}(\epsilon,\delta)$. (Note $X_0=X_{1,\theta_0}$ under this correspondence.)  Next, we pick two smooth cut-off functions $y=f(r)$ and $y=g(\theta)$ whose graphs are given in Figure \ref{fig:cut_off_functions}.\\ 

\begin{figure}[ht]
\begin{center}
\includegraphics{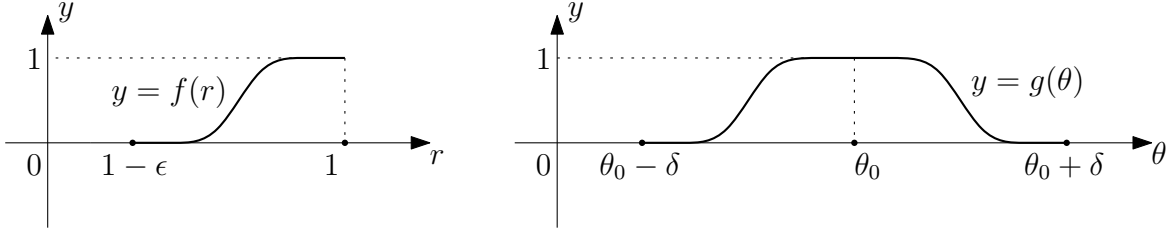}
\caption{Smooth cut-off functions.}
\label{fig:cut_off_functions}
\end{center}
\end{figure}

On the fiber $X_{r,\theta}$, consider the homotopy $F_{t,r,\theta}$, $t \in [0,1]$ of Liouville structures defined by  
$$F_{t,r,\theta}:=(\Psi^{-1}|_{re^{i\theta} \times X_0})^*(\alpha_{f(r)g(\theta)t}).$$ 
Here by abuse of notation we consider $\alpha_{f(r)g(\theta)t}$ on $re^{i\theta} \times X_0 \cong X_0$ for each $(r,\theta)$ and $t$. Note that $F_{t,1,\theta_0}=G_t$, i.e., $F_{t,r,\theta}$ restricts to $G_t$ on $X_0=X_{1,\theta_0}$. Observe that %$$\lim_{r\to 1-\epsilon} F_{t,r,\theta}=\lim_{\theta \to \theta_0\pm \delta} F_{t,r,\theta}=(\Psi^{-1})^*(\alpha_0)$$
%which means that the homotopy 
$F_{t,r,\theta}$ is compactly supported in $\pi^{-1}(A_{\theta_0}(\epsilon,\delta))$.
One can think of $F_{t,r,\theta}$ as the smooth extension of $G_t$ (defined only on $X_{1,\theta_0}$) to the fibers in the subset $\pi^{-1}(A_{\theta_0}(\epsilon,\delta))$, and by tapering using $f(r), g(\theta)$ we gurantee that the Liouville structures of the fibers close to the boundary $$\partial \pi^{-1}(A_{\theta_0}(\epsilon,\delta))= \pi^{-1}(\{r=1-\epsilon\}) \cup \pi^{-1}(\{\theta=\theta_0\pm\delta\})$$ stay constant during the homotopy. \\

Next, using $F_{t,r,\theta}$, we construct a smooth $1$-parameter family of compact convex Lefschetz fibrations on $W$ (such family is called compact convex Lefschetz deformation in \cite{Mc}) as follows: Consider the smooth $1$-parameter families of one- and two-forms on $\pi^{-1}(A_{\theta_0}(\epsilon,\delta))$ given by
$$\tilde{\lambda}_t:= \bigcup_{r,\theta} F_{t,r,\theta} + (\Psi^{-1})^*(f_1 dx_1 + f_2 dx_2+ dR), \quad \tilde{\omega}_t=d \tilde{\lambda}_t, \quad t \in [0,1].$$

Observe that for any fixed time $t \in [0,1]$, the restriction of the pair $(\tilde{\omega}_t,\tilde{\lambda}_t)$ to any fiber $X_{r,\theta}$ is $(dF_{t,r,\theta},F_{t,r,\theta}+(\Psi^{-1})^*(dR)|_{X_{r,\theta}})$, and this defines a convex symplectic structure on $X_{r,\theta}$ (as $R$ vanishes near $A_{\theta_0}(\epsilon,\delta) \times \partial X_0$). Since the triviality condition near its horizontal boundary is also satisfied, the fibration $$(\pi|_{\pi^{-1}(A_{\theta_0}(\epsilon,\delta))},\pi^{-1}(A_{\theta_0}(\epsilon,\delta)),\tilde{\omega}_t,\tilde{\lambda}_t)$$ is a compact convex symplectic fibration for any $t$. Moreover, since $F_{t,r,\theta}$ is compactly supported in $\pi^{-1}(A_{\theta_0}(\epsilon,\delta))$,  we can smoothly extend $\tilde{\omega}_t,\tilde{\lambda}_t$ to smooth $1$-parameter families $\omega'_t,\lambda'_t$ on $W$ by just letting $(\omega'_t,\lambda'_t)=(\omega,\lambda)$ outside $\pi^{-1}(A_{\theta_0}(\epsilon,\delta))$. Hence, we obtain a smooth $1$-parameter family $$(\pi,W,D^2,\omega'_t,\lambda'_t,J_0,j_0)$$ of compact convex Lefschetz fibrations on $W$ such that the smooth family $\lambda'_t$ of one-forms on $W$ restricts to the homotopy $\alpha_t$ of Liouville structures on $X_0$. Note that $$(\pi,W,D^2,\omega'_0,\lambda'_0,J_0,j_0)=(\pi,W,D^2,\omega,\lambda,,J_0,j_0)$$ which is compatible with the Liouville domain $(W,\Omega,\Lambda)$ by assumption.\\

Next, we construct a homotopy of Liouville domains $(W,\Omega_t,\Lambda_t)$ and a compact convex Lefschetz deformation $(\pi,W,\omega_t,\lambda_t,J_0,j_0)$ which is compatible with $(\Omega_t,\Lambda_t)$ for any $t \in [0,1]$: For any fixed $t$, by Theorem \ref{thm:Mc}, there exists a constant $K_t>0$ such that for all $k\geq K_t$ the one-form $\lambda'_t + k\pi^*(\lambda_{D^2})$ defines a Liouville structure on $W$. Let us consider
\begin{center}
$\mathcal{K}:=\{K_t \, | \, t \in [0,1] \}$, $\Lambda'_t:=\lambda'_t + K_t\pi^*(\lambda_{D^2})$, and $\Omega'_t:=d\Lambda'_t$
\end{center}
Observe that, for any $ t \in [0,1]$, $\lambda'_t$ differs from $\lambda'_0=\lambda$ only within a compact set $\pi^{-1}(A_{\theta_0}(\epsilon,\delta))$ which implies that $\Omega'_t$-dual vector field of $\Lambda'_t$ differs from $\Omega$-dual vector field of $\Lambda=\lambda + K\pi^*(\lambda_{D})$ only within the same compact set $\pi^{-1}(A_{\theta_0}(\epsilon,\delta))$. Therefore, $K_t$ is a finite number for any $t$, and hence there exists an upper bound, say $K_{upp}$, for $\mathcal{K}$.
We set $ K'=\textrm{Max} (K,K_{upp})$, and consider two smooth $1$-parameter families:   $$F^1_t:=(W,\lambda + [(1-t)K+tK']\pi^*(\lambda_{D^2})), \quad F^2_t:=(W,\lambda'_t + K'\pi^*(\lambda_{D^2}))$$
for $t \in [0,1]$. As $[(1-t)K+tK']\geq K$ (resp. $K'\geq K_t$) for each $t$, Theorem \ref{thm:Mc} implies that $F^1_t$ (resp. $F^2_t$) is a homotopy of Liouville domains. Now we define two concatenations:

$$ (W,\Lambda_t):=
\begin{cases}
F^1_{2t} & \textrm{if} \quad t \in [0,1/2] \\
F^2_{2t-1} & \textrm{if} \quad t \in [1/2,1]
\end{cases},
$$

$$ (\pi,W,D^2,\omega_t,\lambda_t,J_0,j_0):=
\begin{cases}
(\pi,W,D^2,\omega,\lambda,J_0,j_0) & \textrm{if} \quad t \in [0,1/2] \\
(\pi,W,D^2,\omega'_{2t-1},\lambda'_{2t-1},J_0,j_0) & \textrm{if} \quad t \in [1/2,1]
\end{cases}.
$$
Set $\Omega_t=d\Lambda_t$. Then by construction the Liouville domain $H_t:=(W,\Omega_t,\Lambda_t)$ is compatible with $(\pi,W,D^2,\omega_t,\lambda_t,J_0,j_0)$ for each $t \in [0,1]$. Therefore, we have constructed a compatible compact convex Lefschetz deformation $(\pi,W,\Omega_t,\Lambda_t,X,h)$ such that (i) and (iv) of the proposition hold. Finally, (iii) holds by Theorem \ref{thm:Liouville_Openbooks_Support}, and (ii) follows from Gray Stability Theorem (see, for instance, Theorem 2.2.2 of \cite{Ge}).
\end{proof}

\begin{proof}[\textbf{Proof of Theorem \ref{thm:Legendrian_Realiz_in_convex_Lef_fib}}]
Let $L \subset X_0\setminus \partial X_0$ be any given simply connected Lagrangian submanifold on the page $(X_0,d\beta,\beta)$ of the convex open book $(X,h)$ induced by a compatible compact convex Lefschetz fibration $(\pi,W,\Omega,\Lambda,X,h)$. Here we assume that $$(d\beta,\beta)=(\Omega|_{X_0},\Lambda|_{X_0}).$$
By starting with a similar argument in the proof of Theorem \ref{thm:Legendrian_Realiz_in_Contact_open_book}, we will first construct a homotopy of Liouville structures on $X_0$ as follows: Consider the cotangent bundle $T^*L$ equipped with the symplectic structure $d\lambda_{can}$ where $\lambda_{can}$ %=\Sigma_{i=1}^nq_idp_i$ 
is the canonical $1$-form on $T^*L$. By Lagrangian neighborhood theorem, there exist a neighborhood $N(L_0)$ of the zero section $L_0$ in $T^*L$, a neighborhood $N_{X_0}(L)$ of $L$ in $X_0$, and a diffeomorphism $$\Upsilon:N_{X_0}(L)\to N(L_0)$$ which identically maps $L$ onto $L_0$ and pulls back the symplectic form $d\lambda_{can}|_{N(L_0)}$ to $d\beta|_{N_{X_0}(L)}$. From now on we will denote the restrictions with the same symbols.  Let $\alpha_{can}=\Upsilon^*(\lambda_{can})$. Then since $d\alpha_{can}=\Upsilon^*(d\lambda_{can})=d\beta$, $\beta-\alpha_{can}$ is closed. Therefore, the assumption that $L$ is simply connected implies that there exists a smooth function $g:N_{X_0}(L) \to \R$ such that $\beta-\alpha_{can}=dg$ on $N_{X_0}(L)$. Take a smooth cut-off function $\mu:X_0 \to \R$ which is equal to $1$ near $L$ and identically zero on $X_0 \setminus N_{X_0}(L)$. Then, we can set the equality $\beta-\alpha_{can}=d(\mu g)$ which holds over all $N_{X_0}(L)$. Now consider the smooth $1$-parameter family of one-forms on $X_0$ given by $$\alpha_t:=\beta-td(\mu g), \quad t \in [0,1].$$
Note that for any $t$ we have $d\alpha_t=d\beta$ is symplectic, and so $(d\alpha_t,\alpha_t)$ is an exact symplectic structure on $X_0$ for each $t \in [0,1]$. Let $\chi$ (resp. $\chi_t$) denote the $d\beta$-dual vector field of $\beta$ (resp. $\alpha_t$). That is, $\iota_{\chi}d\beta=\beta$ and $\iota_{\chi_t}d\beta=\alpha_t$. Then $\chi$ is transverse to $\partial X_0$ by assumption. Also $\alpha_t=\beta$ near $\partial W$ as the family $\{\alpha_t\}$ is compactly supported near $L$. Therefore, $\chi_t=\chi$ near $\partial W$ which implies that $\chi_t$ is transverse to $\partial W$ for each $t$. Hence, the family $\{\alpha_t\}$, indeed, defines a homotopy of Liouville domains $$G_t=(X_0,\alpha_{t}), \quad t \in [0,1].$$

Now, by Proposition \ref{prop:Extending_the_homotopy},  there exists a compatible compact convex Lefschetz deformation $(\pi,W,\Omega_t,\Lambda_t,X,h)$, $t \in [0,1]$, such that the smooth $1$-parameter family $(W,\Omega_t,\Lambda_t)$ defines a homotopy of Liouville domains with $(W,\Omega_0,\Lambda_0)=(W,\Omega,\Lambda)$. Set $(W,\Omega',\Lambda')=(W,\Omega_1,\Lambda_1)$. The parts (i), (iii), (iv) of the theorem just follow from the corresponding parts of Proposition \ref{prop:Extending_the_homotopy}, and we have (ii) by Lemma \ref{lem:Liouville_homotopy}. For the part (v) observe that we have $$\Lambda'|_{X_0}=\alpha_1=\beta-d(\mu g)=\alpha_{can}.$$ Moreover, $\alpha_{can}|_L=0$ since $\lambda_{can}|_{L_0}=0$ (the canonical one-form identically vanishes on the zero section of $T^*L$). Hence, we obtain $(\Lambda'|_{\partial W})|_{L}=0$ because $L$ is embedded on the page $X_0$ and the Reeb direction of $\Lambda'|_{\partial W}$ is transverse to $TX_0$. As a result, $L \subset X_0$ is Legendrian in $(\partial W, \textrm{Ker}(\Lambda'|_{\partial W})$. 
This completes the proof of Theorem \ref{thm:Legendrian_Realiz_in_convex_Lef_fib}
\end{proof}

%-----------------------------------------------------------------------------------------------
%-----------------------------------------------------------------------------------------------
%-----------------------------------------------------------------------------------------------

\section{Convex Stabilizations} \label{sec:convex_stabilization}

We first recall positive stabilization process for contact open books (due to Giroux). Let $(M_i, \xi_i)$ be two closed contact manifolds such that each $\xi_i$ is carried by a contact open book $(X_i,h_i)$ on $M_i$. Suppose that $L_i$ is a properly embedded Lagrangian ball in $X_i$ with Legendrian boundary $\partial L_i \subset \partial X_i$. By the Weinstein neighborhood theorem each $L_i$ has a standard neighborhood $N_i$ in $X_i$ which is symplectomorphic to $(T^* D^n, d\lambda_{\textrm{can}})$ where $\lambda_{\textrm{can}}=\textbf{p}\textbf{d}\textbf{q}$ is the canonical $1$-form on $\mathbb{R}^n \times \mathbb{R}^n$ with coordinates $(\textbf{q},\textbf{p})$. Then the \emph{plumbing} or \emph{2-Murasugi sum} $(\mathcal{P}(X_1,X_2;L_1,L_2),h)$ of $(X_1,h_1)$ and $(X_2,h_2)$ \emph{along} $L_1$ and $L_2$ is the open book on the connected sum $M_1 \# M_2$ with the pages obtained by gluing $\Sigma_i$'s together along $N_i$'s by interchanging $\textbf{q}$-coordinates in $N_1$ with $\textbf{p}$-coordinates in $N_2$, and vice versa. To define $h$, extend each $h_i$ to $\tilde{h}_i$ on the new page by requiring $\tilde{h}_i$ to be identity map outside the domain of $h_i$. Then the monodrodmy $h$ is defined to be $\tilde{h}_2 \circ \tilde{h}_1$. Without abuse of notation we will drop the ``tilde'' sign, and write $h=h_2 \circ h_1$.

\begin{definition}[\cite{Gi}] \label{def:Stabilization_Contac_Open Book}
Suppose that $(X,h)$ is a contact open book carrying the contact structure $\xi=\textrm{Ker}(\alpha)$ on a $(2n+1)$-manifold $M$. Let $L$ be a properly embedded Lagrangian $n$-ball in a page $(X,d\alpha)$
such that  $\partial L \subset \partial X$ is a Legendrian
$(n-1)$-sphere in the binding $(\partial X,\alpha |_{\,\partial X})$.
Then the \emph{positive} (or \emph{standard}) \emph{stabilization} $\mathcal{S_{OB}}[(X,h);L]$ of $(\Sigma,h)$
\emph{along} $L$ is the open book $(\mathcal{P}(X,\mathcal{D}(T^*S^n);L,\textbf{D}), \delta \circ h)$ where $\textbf{D}\cong D^n$ is any fiber in the cotangent unit disk bundle $\mathcal{D}(T^*S^n)$ and $\delta:\mathcal{D}(T^*S^n) \to \mathcal{D}(T^*S^n)$ is the right-handed Dehn twist.
\end{definition}

Next, we recall positive stabilization process introduced for Lefschetz fibrations in \cite{AA}.

\begin{definition} [\cite{AA}] \label{def:Stabilization_Lefschetz_Fibration}
Let $(\pi,W,X,h)$ be a (smooth) Lefschetz fibration which induces a contact open book on $\partial W$. Suppose that $\iota:D^n \hookrightarrow (X,\omega)$ is an embeding of a Lagrangian $n$-ball $L=\iota(D^n)$ with a Legendrian boundary $\partial L=\iota(S^{n-1}) \subset \partial X$ on a page of the induced open book. Then the \emph{positive stabilization} $\mathcal{S_{LF}}[(\pi,W,X,h);L]$ of $(\pi,W,X,h)$ \emph{along} $L$ is a Lefschetz fibration $(\pi',W',X',h')$ described as follows:

\begin{itemize}
\item[(I)] $X'$ is obtained from $X$ by attaching a Weinstein $n$-handle
$H=D^n\times D^n$ along the Legendrian sphere $\partial L \subset \partial X$ via the attaching map $\psi:S^{n-1} \times D^n \to \partial X$ such that $\iota^{-1} \circ \psi|_{S^{n-1} \times \{0\}}=\textrm{id}_{S^{n-1}} \in \textrm{Diff}(S^{n-1})$.
\item[(II)] $h'=\delta_{(\phi,\phi')} \circ h$ where $\delta_{(\phi,\phi')}$ is the
right-handed Dehn twist with center $(\phi,\phi')$ (see \cite{K}) defined as follows: $\phi(S^n)$
is the Lagrangian $n$-sphere $S=D^n \times \{0\} \cup_{\partial L} L$ in the
symplectic manifold $(X'=X \cup H, \omega')$ where $\omega'$ is obtained by
gluing $\omega$ and standard symplectic form on $H$. If $\nu_1$ denote the normal bundle of $S$
in $X'$, then the normalization $\phi': TS^n \rightarrow \nu_1$ is given by the
bundle isomorphisms 
\begin{center}
$TS^n \underset{\phi_*}{\stackrel{\cong}{\longrightarrow}} TS \stackrel{\cong}{\longrightarrow} T_SX'/TS = \nu_1.$
\end{center}
\end{itemize}
\end{definition}

In \cite{AA} the authors have also defined the \emph{convex stabilizations} for compatible compact convex Lefschetz fibrations  and convex open books. These are special positive stabilizations along properly embedded Legendrian disks. They have shown that convexly stabilized compatible compact convex Lefschetz fibration (resp. convex open book) is another compatible compact convex Lefschetz fibration (resp. convex open book). In what follows, $\mathcal{S^{C}_{LF}}[(\pi,W,\Omega,\Lambda,X,h);L]$ (resp. $\mathcal{S^{C}_{OB}}[(X,h);L]$) will denote the convex stabilization of the compatible compact convex Lefschetz fibration $(\pi,W,\Omega,\Lambda,X,h)$ (resp. convex open book $(X,h)$) along a properly embedded Legendrian disk $L$. (see \cite{AA} for precise definitions of $\mathcal{S^{C}_{LF}}$ and $\mathcal{S^{C}_{OB}}$.) Let us recall $\omega$-convexity for our next result.

\begin{definition}
(i) A vector field $\chi$ on a smooth manifold $W$ is said to \emph{gradient-like} for a smooth function $\psi: W \to \R$ if $\chi \cdot \psi=\mathcal{L}_{\chi} \psi >0$ away from the critical point of $\psi$.\\ %In such  a case, $\psi$ is called a \emph{Lyapunov function} for $\chi$.\\
(ii) A real-valued function is said to be \emph{exhausting} if it is proper and bounded from below.\\
(iii) An exhausting function $\psi: W \to \R$ on a symplectic manifold $(W,\omega)$ is said to be $\omega$-\emph{convex} if there exists a complete Liouville vector field $\chi$ which is gradient-like for $\psi$.
\end{definition}

\begin{theorem} \label{thm:Convex_stab_gives_Liouville_Lefs_fib}
Let $(\pi,W,\Omega,\Lambda,X,h)$ be a compatible compact convex Lefschetz fibration and $L$ any properly embedded Legendrian disk sitting on a page of the boundary convex open book $(X,h)$. Denote $\mathcal{S^{C}_{LF}}[(\pi,W,\Omega,\Lambda,X,h);L]$ by $(\pi',W',\Omega',\Lambda',X',h')$, and $\Omega'$-dual vector field of $\Lambda'$ by $Z'$. Then there exists a homotopy from $Z'$ to another Liouville vector field $Z''$ of $\Omega'$ which is compactly supported away from $\partial W'$ such that $(W',\Omega',Z'')$ is exact symplectomorphic to a positive expansion $(W,\Omega,Z)_f$ for some nonzero smooth function $f:\partial W\to \R$.
\end{theorem}

\begin{proof}
Suppose $\textrm{dim}(W)=2n+2$. The convex stabilization of $(\pi,W,\omega,\chi,\psi,X,h)$ along $L$ is performed in two stages \cite{AA}: The first one is the subcritical stage where the Weinstein handle $H_n$ (of index $n$) is attached to $W$ along $\partial L \subset \partial X$, say $\tilde{W}=W \cup H_n$, and the second one is the critical stage where the Weinstein handle $H_{n+1}$ (of index $n+1$) is attached to $\tilde{W}$ (resulting $W'=\tilde{W} \cup H_{n+1}$) along certain Legendrian sphere  sitting on a page of the boundary open book obtained in the first stage. As pointed out in Theorem 6.8. of \cite{AA}, we know that $\{H_n,H_{n+1}\}$ is a symplectically canceling pair, more precisely the completion of $(W',\Omega',\Lambda')$ is symplectomorphic to the completion of $(W,\Omega,\Lambda)$. This fact follows from Proposition 12.22 in \cite{CE} or Lemma 3.6b in \cite{E2} (see also Lemma 3.9 in \cite{V}). To obtain an exact  symplectomorphism as claimed we first follow the steps given in the proof of Lemma 3.11 in \cite{V} in the Liouville category. \\

Note that there is a $\Omega$-convex Morse function $\psi:N \to \R$ defined on a colar neighborhood $N$ of $\partial W$. Such $\psi$ can be obtained by integrating $\Omega$-dual (Liouville) vector field $Z$ of $\Lambda$ on $N$. Taking $N$ small enough one can guarantee that $\psi$ has no critical point. The contact boundary $(\partial W, \textrm{Ker}(\iota_Z \Omega|_{\,\partial W}))$ is the (regular) hypersurface of $\psi$ with the highest level, say $c \in \R$. That is, we have $$N=\psi^{-1}(c-\epsilon,c] \quad \textrm{and} \quad \partial W=\psi^{-1}(c).$$ 
In the subcritical stage, we know by \cite{W} that attaching $H_n$ gives rise to a Liouville domain $(\tilde{W},\tilde{\Omega},\tilde{Z})$ such that there exists a $\tilde{\Omega}$-convex Morse function $\tilde{\psi}:N\cup H_n \to \R$ obtained by extending $\psi$ over $H_n$. Note $\tilde{\psi}$ coincides with $\psi$ on $N$ and has a single critical point of index $n$ which is located at the center of $H_n$. Let $p \in \R$ be the corresponding critical value.

\vspace{.05in}

Similarly, by attaching $H_{n+1}$ to $\tilde{W}$ in the critical stage, we get a Liouville domain $(W',\Omega',Z')$ such that there exists a $\Omega'$-convex Morse function $\psi':N\cup H_n \cup H_{n+1} \to \R$ obtained by extending $\tilde{\psi}$ over $H_{n+1}$. This time $\psi'$ coincides with $\tilde{\psi}$ on $N\cup H_n$ and has one more critical point of index $n+1$ which corresponds to the center of $H_{n+1}$. Let $q \in \R$ be the corresponding critical value. Suppose that $d \in \R$ is the absolute maximum value of $\psi'$, or equivalently, $$N\cup H_n \cup H_{n+1}=\psi'^{-1}(c-\epsilon,d] \quad \textrm{and} \quad \partial W'=\psi'^{-1}(d).$$

\vspace{.05in}

Note that $c<p<q<d$. Also there is a unique flow line, say $l$, of $Z'$ joining $p$ and $q$, and so the belt sphere of $H_n$ intersects the attaching sphere of $H_{n+1}$ transversely once. Therefore, Proposition 12.22 \cite{CE} implies that we can deform $\psi'$ to another $\Omega'$-convex Morse function $\psi''$ (with gradient-like Liouville vector field $Z''$) on $N\cup H_n \cup H_{n+1}$ such that $(\psi'',Z'')$ coincides with $(\psi',Z')$ outside a neighborhood of $l$ and $\psi''$ has no critical value greater than $c$. This means that the cobordism $\psi''^{-1}[c,d]$ together with the restricted Liouville structure is exact symplectomorphic to some compact subset of the positive part of the symplectization of $(\partial W, \textrm{Ker}(\iota_Z \Omega|_{\,\partial W}))$. In fact, by following the trajectories of $Z''$, one can realize $\psi''^{-1}(d)$ as the graph of a nonzero smooth function on $\partial W \times \{c\}$. So there exists an exact symplectomorphism $\Upsilon_0:W_f^{cob} \to \psi''^{-1}[c,d]$ for some nonzero smooth function $f:\partial W\to \R$.

Also by construction $(W,\Omega,Z)$ is an embedded Liouville subdomain of $(W',\Omega',Z'')$. As a result, we can define a diffeomorphism $\Upsilon:W \cup W_f^{cob} \longrightarrow W'=W \cup \psi''^{-1}[c,d]$ by the rule
$$ \Upsilon(x)=
\begin{cases}
\quad x & \textrm{if} \quad x \in W \\
\Upsilon_0(x) & \textrm{if} \quad x \in W_f^{cob}
\end{cases}
$$
Clearly, this is an exact symplectomorphism between the $f$-positive expansion  $(W,\Omega,Z)_f$ and $(W',\Omega',Z'')$.
\end{proof}

We close the paper by recalling some facts about convex stabilizations in the terminology of the present paper. All of the next three results have exact symplectic versions in \cite{AA}. For instance, the first one is a special case of Theorem 6.7 in \cite{AA}. Here we provide the proofs only for the last two.

\begin{theorem} \label{thm:Liouville_openbooks_Liouville_Lefs.fibs}
$\mathcal{S^{C}_{LF}}[(\pi,W,\Omega,\Lambda,X,h);L]$ induces the open book $\mathcal{S^{C}_{OB}}[(X,h);L]$. Conversely, if a convex open book $(X,h)$ is induced by $(\pi,W,\Omega,\Lambda,X,h)$, then any convex stabilization $\mathcal{S^{C}_{OB}}[(X,h);L]$ of $(X,h)$ is induced by the convex stabilization $\mathcal{S^{C}_{LF}}[(\pi,W,\Omega,\Lambda,X,h);L]$. \qed
\end{theorem}

\begin{corollary} \label{cor:convex_stab_gives_contactomorphic_mfld}
Let $(\pi',W',\Omega',\Lambda',X',h')$ be any convex stabilization of a compatible compact convex Lefschetz fibration 
$(\pi,W,\Omega,\Lambda,X,h) $. 
Let $\xi=\emph{Ker}(\Lambda|_{\partial W})$ (resp. $\xi'=\emph{Ker}(\Lambda'|_{\partial W'})$) be the induced contact structure on $\partial W'$ (resp. $\partial W'$). Then $(\partial W,\xi)$ is contactomorphic to $(\partial W',\xi')$.
\end{corollary}

\begin{proof}
By Theorem \ref{thm:Convex_stab_gives_Liouville_Lefs_fib}, there exists another Liouville structure $(\Omega',Z'')$ on $W'$ such that $(W',\Omega',Z'')$ is exact symplectomorphic to $(W,\Omega,Z)_f$ for some $f:\partial W \to \R_{>0}$. Note that the deformation of $Z'$ to $Z''$ is performed away from the boundary $\partial W'$. Therefore, the contact boundary of $(W',\Omega',Z'')$ is still $(\partial W', \xi')$. Denote the Liouville domain $(W,\Omega,Z)_f$ by $(\tilde{W},\tilde{\Omega},\tilde{Z})$ and the induced contact structure on $\partial \tilde{W}$ by $\tilde{\xi}$. Then $(\partial \tilde{W},\tilde{\xi})$ is contactomorphic to $(\partial W', \xi')$ by Theorem \ref{thm:Convex_stab_gives_Liouville_Lefs_fib}. Now the proof follows from the fact that $(\partial W, \xi)$ and $(\partial \tilde{W},\tilde{\xi})$ are contactomorphic (This was explained in Remark \ref{rem:no_symplecto_but_contacto}).
\end{proof}

\begin{corollary} \label{cor:convex_stab_respects_contact_str}
Let $\xi$ (as in Theorem \ref{thm:Liouville_Openbooks_Support}) be the contact structure carried by a convex open book $(X,h)$. Then any convex stabilization $\mathcal{S^{C}_{OB}}[(X,h);L]$ of $(X,h)$ carries $\xi$.
\end{corollary}

\begin{proof} By assumption, there is a compatible compact convex Lefschetz fibration $(\pi,W,\Omega,\Lambda,X,h)$ which induces $(X,h)$, and the Liouville structure $(\Omega,\Lambda)$ induces $\xi$ on $\partial W$. Theorem \ref{thm:Liouville_openbooks_Liouville_Lefs.fibs} implies that $\mathcal{S^{C}_{OB}}[(X,h);L]$ is induced by $\mathcal{S^{C}_{LF}}[(\pi,W,\Omega,\Lambda,X,h);L]$. Moreover, by Theorem \ref{thm:Liouville_Openbooks_Support}, $\mathcal{S^{C}_{OB}}[(X,h);L]$ carries the contact structure induced by the Liouville structure on the total space of $\mathcal{S^{C}_{LF}}[(\pi,W,\Omega,\Lambda,X,h);L]$. Now the proof follows from Corollary \ref{cor:convex_stab_gives_contactomorphic_mfld}.
\end{proof}

%-----------------------------------------------------------------------------------------------
%-----------------------------------------------------------------------------------------------
%-----------------------------------------------------------------------------------------------

\addcontentsline{toc}{chapter}{\textsc{References}}


\addcontentsline{TOC}{chapter}{References}
\begin{thebibliography}{99}

\bibitem[1]{AA} S. Akbulut and M. F. Arikan,
{\em Stabilizations via Lefschetz Fibrations and Exact Open Books}, arXiv:1112.0519.

\bibitem[2]{AO} S. Akbulut and B. Ozbagci,
{\em Lefschetz fibrations on compact Stein surfaces}, Geom. Topol. 5
(2001), 319-334.

\bibitem[3]{CE} K. Cieliebak and Y. Eliashberg,
{\em Symplectic geometry of affine complex manifolds}
(Incomplete draft, November 2012)

\bibitem[4]{E2} Y. Eliashberg, 
{\em Symplectic geometry of plurisubharmonic functions}, Gauge theory and symplectic geometry (Montreal, PQ, 1995), NATO Adv. Sci. Inst. Ser. C Math. Phys. Sci., vol. 488, Kluwer Acad. Publ., Dordrecht, 1997, With notes by Miguel Abreu, pp. 49--67. MR 1461569 (98g:58055)

\bibitem[5]{EG} Y. Eliashberg and M. Gromov,
{\em Convex Symplectic Manifolds}, Proceedings of Symposia in Pure Mathematics, vol. 52, Part 2, 135-162 (1991).

\bibitem[6] {Ge} H. Geiges,
{\em An Introduction to Contact Topology}, Cambridge University Press, (2008).

\bibitem[7]{Gi} E. Giroux, {\em G\'eom\'etrie de contact: de la dimension trois vers les dimensions sup\'erieures}, Proceedings of the International Congress of Mathematicians, Vol. II (Beijing), Higher Ed. Press,
(2002), pp. 405-414. MR 2004c:53144

\bibitem[8]{K} A. Kas,
{\em On the handlebody decomposition associated to a Lefschetz fibration}, Pacific Journal of Mathematics 89, No. 1 (1980): 89--104

\bibitem[9]{Mc} M. McLean, {\em Lefschetz fibrations and symplectic homology}, Geom. Topol., 13 (2009), 1877-1944.

\bibitem[10]{S1} P. Seidel, {\em Vanishing cycles and mutation}, European Congress of Mathematics, Vol. II (Barcelona, 2000), 65-85, Progr. Math., 202, Birkh\"auser, Basel, 2001.

\bibitem[11]{S2} P. Seidel, {\em A long exact sequence for symplectic Floer cohomology}, Topology 42 (2003), p. 1003-1063.

\bibitem[12]{S3} P. Seidel, {\em A biased view of symplectic cohomology, Current developments
in mathematics}, 2006, 211-253, Int. Press (2008).

\bibitem[13]{V} O. Van Koert, {\em Lecture notes on stabilization of contact open books},    (arXiv:1012.4359v1)

\bibitem[14]{W} A. Weinstein,
{\em Contact surgery and symplectic handlebodies}, Hokkaido Math. J. 20 (1991),
no. 2, 241-251.


\end{thebibliography}
\end{document}